\documentclass[11pt,english,oneside]{amsart}
\usepackage[T1]{fontenc}
\usepackage[latin9]{inputenc}
\usepackage{geometry}
\usepackage{amssymb}
\usepackage{color}

\usepackage{graphicx,color}
\usepackage{amsmath, amssymb, graphics}

\usepackage{graphicx}

\makeatletter
\numberwithin{equation}{section} %% Comment out for sequentially-numbered
\numberwithin{figure}{section} %% Comment out for sequentially-numbered
  \@ifundefined{theoremstyle}{\usepackage{amsthm}}{}
  \theoremstyle{plain}
  \newtheorem{thm}{Theorem}[section]
  \theoremstyle{plain}
  
  \theoremstyle{plain}
  
  \theoremstyle{Remark}
  \newtheorem{rem}[thm]{Remark}
  \theoremstyle{remark}
  
  \theoremstyle{plain}
  \newtheorem{lem}[thm]{Lemma}

\usepackage{geometry}

\newcommand{\R}{\mathbb{R}}

\def\com#1{ \hbox{#1}}

\smallskip
\def\<{{\langle }}
\def\>{{\rangle }}

\makeatother

\usepackage{babel}

\def\com#1{ \quad\hbox{#1}\quad}

\def\R{\hbox{\bf R}}
\def\Sp{\hbox{\bf S}}
\def\Hi{\hbox{\bf H}}

\smallskip
\def\<{{\langle }}
\def\>{{\rangle }}

\makeatother

\begin{document}

\title[cmc hypersurfaces]{CMC hypersurfaces on Riemannian and Semi-Riemannian manifolds}

\author{ Oscar M. Perdomo }

\date{\today}

\curraddr{Department of Mathematics\\
Central Connecticut State University\\
New Britain, CT 06050\\
}

\email{ perdomoosm@ccsu.edu}

\begin{abstract}

In this paper we generalize the explicit formulas for cmc immersion given in \cite{P} and \cite{P1} of hypersurfaces of Euclidean spaces, spheres and hyperbolic spaces to provide explicit examples of several families of immersions with constant mean curvature and non constant principal curvatures, in semi-riemannian manifolds with constant sectional curvature. In particular, we prove that every $h\in[-1,-\frac{2\sqrt{n-1}}{n})$ can be realized as  the constant curvature of a complete immersion of  $S_1^{n-1}\times \R$ in the $(n+1)$-dimensional de Sitter space $\Sp_1^{n+1}$. We provide 3 types of immersions with cmc in the Minkowski space, 5 types of immersion with cmc in the de Sitter space and 5 types of immersion with cmc in the anti de Sitter space.  At the end of the paper we analyze the families of examples that can be extended to closed hypersurfaces.

\end{abstract}


\maketitle

\section{Introduction and preliminaries}

For any non negative integer $k\le n+2$, we will denote by $\R_k^{n+2}$ the space $\R^{n+2}$ endowed with the metric

$$\<v,w\>=-v_1w_1-\dots-v_kw_k+v_{k+1}w_{k+1}+\dots+v_{n+2}w_{n+2}$$

We also define

$$\Sp_{k}^{n+1}=\{x\in\R_k^{n+2}:\<x,x\>=1 \} $$

and when $k\ge1$,

$$\Hi_{k-1}^{n+1}=\{x\in\R_k^{n+2}:\<x,x\>=-1 \}  $$

with the metric induced by $\R_k^{n+1}$. Notice that $\Hi_0^{n}$, $\Sp_0^n$ and $\R_0^n$ are the $n$-dimensional hyperbolic space, sphere and  the Euclidean space respectively.

In this paper we will modify the explicit formulas for cmc hypersurfaces of the Euclidean space, spheres and hyperbolic spaces given in \cite{P} and \cite{P1} to  provide explicit formulas for cmc hypersurfaces in $\R_k^{n+1}$, $\Hi_k^{n+1}$ and $\Sp_k^{n+1}$. Several of these examples are well known. The references at the end provide part of the history and context of the given examples. Since several of the examples in semi riemannian manifold are given by the same ODE as in the riemannian case, for these examples, the nice geometric interpretations given by Sterling in \cite{SI} hold also true.  The author has been informed that Professor Bingye Wu has classified all (space like) hypersurfaces in (Lorentzian) space forms with constant generalized m-th curvature and two distinct principal curvatures.

\section{CMC hypersurfaces}

Let us start with three basic observations that will be used in all the examples.

\begin{lem}\label{mu}
Let $g:(a,b) \to \R $ be a smooth positive function, $h$ any real number, $c$  a non zero real number  and $n$ an integer greater than $1$. If

$$\lambda=h+g^{-n},\quad \mu=h-(n-1)g^{-n}\com{and} r=\frac{g}{\sqrt{|c|}}$$

then,

$$(n-1)\lambda+\mu=nh \com{and} (r\, \lambda)^\prime=\mu r^\prime$$

\end{lem}

\begin{proof}
First notice that $\lambda-\mu=ng^{-n}$. Then

$$\lambda^\prime=-ng^{-n}\frac{g^\prime}{g}=-(\lambda-\mu)\frac{r^\prime}{r}$$

The equation above implies the lemma
\end{proof}

\begin{lem}\label{solution lemma 1}
Let $q:[a,\infty)\to \R$ be a smooth function such that $q(a)=0$, $\lim_{t\to\infty}\frac{q(t)}{t^2}=\epsilon >0$,  $\, q^\prime(a)>0$ and $q(t)>0$ for all $t>a$. If

$$F:[a,\infty)\to \R \com{is given by} F(t)=\int_a^t\frac{1}{\sqrt{q(\tau)}}d\tau $$

then, $F(t)$ is a well-defined, strictly increasing function and $\lim_{t\to\infty}F(t)=\infty$. Moreover, if $G:[0,\infty)\to [a,\infty)$ is the inverse function of $F$, then, $g:(-\infty,\infty)\to [a,\infty)$ given by

$$g(t)=G(t)\quad\com{if}t\ge0 \com{and} g(t)=G(-t)\quad\com{if}t<0 $$

is a solution of the ordinary differential equation $(g^\prime(t))^2=q(g(t))$.

\end{lem}

\begin{proof}
In order to prove that the limit when $t\to \infty$ of $F(t)$ is infinity it is enough to compare $F(t)$ with $\log(t)=\int_1^t\frac{1}{\sqrt{\tau^2}}d\tau$ for big positive values of $t$. A direct verification proves the rest of the lemma.
\end{proof}

\begin{lem} \label{solutions lemma 2}
Let $q:[a,b]\to \R$ be a smooth function such that $q(a)=0=q(b)$,  $\, q^\prime(a)>0$, $\, q^\prime(b)<0$  and $q(t)>0$ for all $a<t<b$. If

$$F:[a,b]\to \R \com{is given by} F(t)=\int_a^t\frac{1}{\sqrt{q(\tau)}}d\tau $$

then, $F(t)$ is a well-defined strictly increasing function and $\lim_{t\to b}F(t)=\frac{T}{2}<\infty$. Moreover, if $G:[0,\frac{T}{2}]\to [a,b]$ is the inverse function of $F$, then, the $T$-periodic function $g:(-\infty,\infty)\to [a,b]$ that satisfies

$$g(t)=G(t)\quad\com{if}0\le t\le \frac{T}{2} \com{and} g(t)=G(-t)\quad\com{if}-\frac{T}{2}\le t\le 0 $$

is a solution of the ordinary differential equation $(g^\prime(t))^2=q(g(t))$.

\end{lem}

\subsection{CMC hypersurfaces in $\Sp_k^{n+1}$}\label{Sitter examples}

Let us start this section by considering the following differential equations

\begin{eqnarray}\label{Sphere m1}
(g^\prime(t))^2=q_1(g(t))\com{where} q_1(t)=c-t^2+t^2(h+t^{-n})^2
\end{eqnarray}

and
\begin{eqnarray}\label{Sphere 1}
(g^\prime(t))^2=p_1(g(t))\com{where} p_1(t)=c-t^2-t^2(h+t^{-n})^2
\end{eqnarray}

where $c$ is a constant.

\begin{thm}\label{examples1 Sm1}
Let $c$ be a positive constant and let $g:(a_1,a_2)\to \R$ be a non constant solution of (\ref{Sphere m1}) such that $g(t)>\sqrt{c}$. For any pair of integers $k$ and $n$  such that $n\ge 2$ and $1\le k\le n$, if we define

$$\lambda=h+g^{-n},\quad \mu=h-(n-1)g^{-n}\quad r=\frac{g}{\sqrt{c}} \com{and} \theta(t)=\int_0^t\frac{r(\tau)\lambda(\tau)}{r(\tau)^2-1}d\tau$$

and

$$\tilde{\Sp}_{k-1}^{n-1}=\Sp_k^{n+1}\cap\tilde{\R}_{k-1}^n\com{where}\tilde{\R}_{k-1}^n=\{x\in\R_k^{n+2}:x_k=x_{k+1}=0\} $$

$$ B_2(t)=(0,\dots,0,\cosh(t),\sinh(t),0,\dots,0)\com{and} B_3(t)=(0,\dots,0,\sinh(t),\cosh(t),0,\dots,0) $$

then, the map $\phi:\tilde{\Sp}_{k-1}^{n-1}\times (a_1,a_2)\to \Sp_k^{n+1}$ given by

$$\phi(y,u)=r(u)\, y+\sqrt{r^2(u)-1} \, B_2(\theta(u))$$

defines a hypersurface with constant mean curvature h.

\end{thm}

\begin{proof}
A direct computation shows the following identities,

$$ (r^\prime)^2-r^2 \lambda^2=1-r^2,\com{and} \lambda r^\prime +r\lambda^\prime=\mu r^\prime$$

Notice that

$$\<y,y\>=1,\quad\<B_2,B_2\>=-1,\quad \<B_3,B_3\>=1,\quad \<B_2,B_3\>=0,\quad B_2^\prime=\frac{r\lambda}{r^2-1}B_3$$

A direct verification shows that $\<\phi,\phi\>=1$ and that

$$\frac{\partial{\phi}}{\partial u}=r^\prime\, y+\frac{rr^\prime}{\sqrt{r^2-1}} \, B_2+
\frac{r \lambda}{\sqrt{r^2-1}} \, B_3$$

satisfies $\<\frac{\partial{\phi}}{\partial u},\frac{\partial{\phi}}{\partial u}\>=1$.
We have that the tangent space of the immersion at $\phi(y,t)$ is given by

$$T_{\phi(y,u)}=\{v+s\, \frac{\partial \phi}{\partial u}: v\in \tilde{\R}_{k-1}^n,\quad \<v,y\>=0\com{and} s\in{\bf R}\}$$

A direct verification shows that the map

$$\nu=-r\lambda\,   y -\frac{r^2\, \lambda}{\sqrt{r^2-1}} \, B_2 - \frac{r^\prime}{\sqrt{r^2-1}} \, B_3$$

satisfies that $\<\nu,\nu\>=-1$, $\<\nu,\frac{\partial \phi}{\partial u}\>=0$ and, for any $v\in\tilde{\R}_{k-1}^n$ with $\<v,y\>=0$, we have that $\<\nu,v\>=0$. It then follows that $\nu$ is a Gauss map of the immersion $\phi$. The fact that the immersion $\phi$ has constant mean curvature $h$ follows because, for any unit vector $v$ in $\tilde{\R}_{k-1}^n$ perpendicular to $y$, we have that

$$\beta(t)= r\cosh(t)\, y+r\sinh(t)\, v +\sqrt{r^2-1}\, B_2=\phi(\cosh(t)y+\sinh(t)v,u)\com{in the case that $\<v,v\>=-1$} $$

and

$$\beta(t)= r\cos(t)\, y+r\sin(t)\, v +\sqrt{r^2-1}\, B_2=\phi(\cos(t)y+\sin(t)v,u)\com{in the case that $\<v,v\>=1$} $$

satisfies that $\beta(0)=\phi(y,u)$, $\beta^\prime(0)=rv$ and

$$\frac{d \nu(\beta(t))}{dt}\big{|}_{t=0} = d\nu(rv)=-r\lambda\, v$$

Therefore,  $\lambda$ is a principal curvature with multiplicity $n-1$. Now, since $\<\frac{\partial \nu}{\partial u},v\>=0$ for every $v\in T_{\phi(y,u)}\cap \tilde{\R}_{k-1}^n$, we have that $\frac{\partial \phi}{\partial u}$ defines a principal direction, i.e. we must have that $\frac{\partial \nu}{\partial u}$ is a multiple of $\frac{\partial \phi}{\partial u}$.  A direct verification using the lemma (\ref{mu}) shows that,

$$\<\frac{\partial \nu}{\partial u},y\>=-\lambda^\prime \, r-\lambda r^\prime=-\mu\, r^\prime=-(nh-(n-1)\lambda)r^\prime$$

We also have that $\<\frac{\partial \phi}{\partial u},y\>=r^\prime$, therefore, since $r$ is not constant, we obtain that

$$\frac{\partial \nu}{\partial u}= d\nu(\frac{\partial \phi}{\partial u})=-\mu\, \frac{\partial \phi}{\partial u}=-(nh-(n-1)\lambda) \frac{\partial \phi}{\partial u}$$

It follows that the other principal curvature is $nh-(n-1)\lambda$. Therefore, $\phi$ defines an immersion with constant mean curvature $h$, this proves the theorem.

\end{proof}

\begin{thm}\label{rem example1 Sm1}
With the same notation as in the previous theorem, if $r_0>1$ is any constant, then the map $\phi:\tilde{\Sp}_{k-1}^{n-1}\times \R \to \Sp_k^{n+1}$ given by

$$\phi(y,t)=r_0\, y+\sqrt{r_0^2-1} \, B_2(t)$$

defines a complete hypersurface with constant mean curvature

$$ \frac{r_0}{n\, \sqrt{r_0^2-1}}\, +\, \frac{(n-1)\, \sqrt{r_0^2-1}}{n\, r_0}$$

In the case $k=1$, these examples are known as hyperbolic cylinder in the de Sitter space $\Sp_1^{n+1}$ and they provide space like complete hypersurfaces with constant mean curvature values in $[\frac{2\, \sqrt{n-1}}{n },\infty)$ when $n>2$ and $(1,\infty)$ when $n=2$.

\end{thm}

\begin{thm}\label{examples2 Sm1}
Let $c$ be a negative constant and let $g:(a_1,a_2)\to \R$ be a positive non constant solution of (\ref{Sphere m1}). For any pair of integers $k$ and $n$  such that $n\ge 2$ and $1\le k\le n$, if we define

$$\lambda=h+g^{-n},\quad \mu=h-(n-1)g^{-n}\com{and} r=\frac{g}{\sqrt{-c}}\quad \theta(t)=\int_0^t\frac{r(\tau)\lambda(\tau)}{r(\tau)^2+1}d\tau$$

and

$$\tilde{\Hi}_{k-1}^{n-1}=\Hi_{k-1}^{n+1}\cap\tilde{\R}_{k}^n\com{where}\tilde{\R}_{k}^n=\{x\in\R_k^{n+2}:x_{n+1}=x_{n+2}=0\} $$

$$ B_2(t)=(0,\dots,0,\cos(t),\sin(t))\com{and} B_3(t)=(0,\dots,0,-\sin(t),\cos(t)) $$

then, the map $\phi:\tilde{\Hi}_{k-1}^{n-1}\times (a_1,a_2)\to \Sp_k^{n+1}$ given by

$$\phi(y,u)=r(u)\, y+\sqrt{r^2(u)+1} \, B_2(\theta(u))$$

defines a hypersurface with constant mean curvature h.

\end{thm}

\begin{proof}
A direct computation shows the following identities,

$$ (r^\prime)^2-r^2 \lambda^2=-1-r^2,\com{and} \lambda r^\prime +r\lambda^\prime=\mu r^\prime$$

Notice that

$$\<y,y\>=-1,\quad\<B_2,B_2\>=1,\quad \<B_3,B_3\>=1,\quad \<B_2,B_3\>=0,\quad B_2^\prime=\frac{r\lambda}{r^2+1}B_3$$

A direct verification shows that $\<\phi,\phi\>=1$ and that

$$\frac{\partial{\phi}}{\partial u}=r^\prime\, y+\frac{rr^\prime}{\sqrt{r^2+1}} \, B_2+
\frac{r \lambda}{\sqrt{r^2+1}} \, B_3$$

satisfies $\<\frac{\partial{\phi}}{\partial u},\frac{\partial{\phi}}{\partial u}\>=1$.
We have that the tangent space of the immersion at $\phi(y,t)$ is given by

$$T_{\phi(y,u)}=\{v+s\, \frac{\partial \phi}{\partial u}: v\in \tilde{\R}_{k}^n,\quad \<v,y\>=0\com{and} s\in{\bf R}\}$$

A direct verification shows that the map

$$\nu=-r\lambda\,   y -\frac{r^2\, \lambda}{\sqrt{r^2+1}} \, B_2 - \frac{r^\prime}{\sqrt{r^2+1}} \, B_3$$

satisfies that $\<\nu,\nu\>=-1$, $\<\nu,\frac{\partial \phi}{\partial u}\>=0$ and, for any $v\in\tilde{\R}_{k-1}^n$ with $\<v,y\>=0$, we have that $\<\nu,v\>=0$. It then follows that $\nu$ is a Gauss map of the immersion $\phi$. The fact that the immersion $\phi$ has constant mean curvature $h$ follows because, for any unit vector $v$ in $\tilde{\R}_{k-1}^n$ perpendicular to $y$, we have that

$$\beta(t)= r\cosh(t)\, y+r\sinh(t)\, v +\sqrt{r^2+1}\, B_2=\phi(\cosh(t)y+\sinh(t)v,u)\com{in the case that $\<v,v\>=1$} $$

and

$$\beta(t)= r\cos(t)\, y+r\sin(t)\, v +\sqrt{r^2+1}\, B_2=\phi(\cos(t)y+\sin(t)v,u)\com{in the case that $\<v,v\>=-1$} $$

satisfies that $\beta(0)=\phi(y,u)$, $\beta^\prime(0)=rv$ and

$$\frac{d \nu(\beta(t))}{dt}\big{|}_{t=0} = d\nu(rv)=-r\lambda\, v$$

Therefore,  $\lambda$ is a principal curvature with multiplicity $n-1$. Now, since $\<\frac{\partial \nu}{\partial u},v\>=0$ for every $v\in T_{\phi(y,u)}\cap \tilde{\R}_{k-1}^n$, we have that $\frac{\partial \phi}{\partial u}$ defines a principal direction, i.e. we must have that $\frac{\partial \nu}{\partial u}$ is a multiple of $\frac{\partial \phi}{\partial u}$.  A direct verification using the lemma (\ref{mu}) shows that,

$$\<\frac{\partial \nu}{\partial u},y\>=-\lambda^\prime \, r-\lambda r^\prime=-\mu\, r^\prime=-(nh-(n-1)\lambda)r^\prime$$

We also have that $\<\frac{\partial \phi}{\partial u},y\>=r^\prime$, therefore, since $r$ is not constant, we obtain that

$$\frac{\partial \nu}{\partial u}= d\nu(\frac{\partial \phi}{\partial u})=-\mu\, \frac{\partial \phi}{\partial u}=-(nh-(n-1)\lambda) \frac{\partial \phi}{\partial u}$$

It follows that the other principal curvature is $nh-(n-1)\lambda$. Therefore $\phi$ defines an immersion with constant mean curvature $h$, this proves the theorem.

\end{proof}

\begin{thm}\label{rem example2 Sm1}
With the same notation as in the previous theorem, if $r_0$ is any constant, then, the map $\phi:\tilde{\Hi}_{k-1}^{n-1}\times \R \to \Sp_k^{n+1}$ given by

$$\phi(y,t)=r_0\, y+\sqrt{r_0^2+1} \, B_2(t)$$

defines a complete hypersurface with constant mean curvature

$$ \frac{r_0}{n\, \sqrt{r_0^2+1}}\, +\, \frac{(n-1)\, \sqrt{r_0^2+1}}{n\, r_0}$$

These examples provide complete space like hypersurfaces with constant mean curvature values in $(1,\infty)$.
%$[\frac{n^2-2 n+2}{n\, \sqrt{n-1}},\infty)$.

\end{thm}

\begin{thm}\label{examples3 Sm1}
Let $c$ be a positive constant and let $g:(a_1,a_2)\to \R$ be a non constant positive solution of (\ref{Sphere m1}) such that $g(t)<\sqrt{c}$. For any pair of integers $k$ and $n$  such that $n\ge 2$ and $1\le k\le n$, if we define

$$\lambda=h+g^{-n},\quad \mu=h-(n-1)g^{-n}\com{and} r=\frac{g}{\sqrt{c}}\quad \theta(t)=\int_0^t\frac{r(\tau)\lambda(\tau)}{1-r(\tau)^2}d\tau$$

and

$$\tilde{\Sp}_{k-1}^{n-1}=\Sp_k^{n+1}\cap\tilde{\R}_{k-1}^n\com{where}\tilde{\R}_{k-1}^n=\{x\in\R_k^{n+2}:x_k=x_{k+1}=0\} $$

$$ B_2(t)=(0,\dots,0,\sinh(t),\cosh(t),0,\dots,0)\com{and} B_3(t)=(0,\dots,0,\cosh(t),\sinh(t),0,\dots,0) $$

then, the map $\phi:\tilde{\Sp}_{k-1}^{n-1}\times (a_1,a_2)\to \Sp_k^{n+1}$ given by

$$\phi(y,u)=r(u)\, y+\sqrt{1-r^2(u)} \, B_2(\theta(u))$$

defines a hypersurface with constant mean curvature h.

\end{thm}

\begin{proof}
A direct computation shows the following identities,

$$ (r^\prime)^2-r^2 \lambda^2=1-r^2,\com{and} \lambda r^\prime +r\lambda^\prime=\mu r^\prime$$

Notice that

$$\<y,y\>=1,\quad\<B_2,B_2\>=1,\quad \<B_3,B_3\>=-1,\quad \<B_2,B_3\>=0,\quad B_2^\prime=\frac{r\lambda}{1-r^2}B_3$$

A direct verification shows that $\<\phi,\phi\>=1$ and that

$$\frac{\partial{\phi}}{\partial u}=r^\prime\, y-\frac{rr^\prime}{\sqrt{1-r^2}} \, B_2+
\frac{r \lambda}{\sqrt{1-r^2}} \, B_3$$

satisfies $\<\frac{\partial{\phi}}{\partial u},\frac{\partial{\phi}}{\partial u}\>=1$.
We have that the tangent space of the immersion at $\phi(y,t)$ is given by

$$T_{\phi(y,u)}=\{v+s\, \frac{\partial \phi}{\partial u}: v\in \tilde{\R}_{k-1}^n,\quad \<v,y\>=0\com{and} s\in{\bf R}\}$$

A direct verification shows that the map

$$\nu=-r\lambda\,   y +\frac{r^2\, \lambda}{\sqrt{1-r^2}} \, B_2 - \frac{r^\prime}{\sqrt{1-r^2}} \, B_3$$

satisfies that $\<\nu,\nu\>=-1$, $\<\nu,\frac{\partial \phi}{\partial u}\>=0$ and, for any $v\in\tilde{\R}_{k-1}^n$ with $\<v,y\>=0$, we have that $\<\nu,v\>=0$. It then follows that $\nu$ is a Gauss map of the immersion $\phi$. The fact that the immersion $\phi$ has constant mean curvature $h$ follows because, for any unit vector $v$ in $\tilde{\R}_{k-1}^n$ perpendicular to $y$, we have that

$$\beta(t)= r\cosh(t)\, y+r\sinh(t)\, v +\sqrt{1-r^2}\, B_2=\phi(\cosh(t)y+\sinh(t)v,u)\com{in the case that $\<v,v\>=-1$} $$

and

$$\beta(t)= r\cos(t)\, y+r\sin(t)\, v +\sqrt{1-r^2}\, B_2=\phi(\cos(t)y+\sin(t)v,u)\com{in the case that $\<v,v\>=1$} $$

satisfies that $\beta(0)=\phi(y,u)$, $\beta^\prime(0)=rv$ and

$$\frac{d \nu(\beta(t))}{dt}\big{|}_{t=0} = d\nu(rv)=-r\lambda\, v$$

Therefore,  $\lambda$ is a principal curvature with multiplicity $n-1$. Now, since $\<\frac{\partial \nu}{\partial u},v\>=0$ for every $v\in T_{\phi(y,u)}\cap \tilde{\R}_{k-1}^n$, we have that $\frac{\partial \phi}{\partial u}$ defines a principal direction, i.e. we must have that $\frac{\partial \nu}{\partial u}$ is a multiple of $\frac{\partial \phi}{\partial u}$.  A direct verification using the lemma (\ref{mu}) shows that,

$$\<\frac{\partial \nu}{\partial u},y\>=-\lambda^\prime \, r-\lambda r^\prime=-\mu\, r^\prime=-(nH-(n-1)\lambda)r^\prime$$

We also have that $\<\frac{\partial \phi}{\partial u},y\>=r^\prime$, therefore, since $r$ is not constant, we obtain that

$$\frac{\partial \nu}{\partial u}= d\nu(\frac{\partial \phi}{\partial u})=-\mu\, \frac{\partial \phi}{\partial u}=-(nH-(n-1)\lambda) \frac{\partial \phi}{\partial u}$$

It follows that the other principal curvature is $nH-(n-1)\lambda$. Therefore $\phi$ defines an immersion with constant mean curvature $H$, this proves the theorem.

\end{proof}

\begin{thm}\label{examples1 S1}
Let $c$ be a positive constant and let $g:(a_1,a_2)\to \R$ be a non constant positive solution of (\ref{Sphere 1}) such that $g(t)<\sqrt{c}$. For any pair of integers $k$ and $n$  such that $n\ge 2$ and $0\le k\le n$, if we define

$$\lambda=h+g^{-n},\quad \mu=h-(n-1)g^{-n}\com{and} r=\frac{g}{\sqrt{c}}\quad \theta(t)=\int_0^t\frac{r(\tau)\lambda(\tau)}{1-r(\tau)^2}d\tau$$

and

$$\tilde{\Sp}_{k}^{n-1}=\Sp_k^{n+1}\cap\tilde{\R}_{k}^n\com{where}\tilde{\R}_{k-1}^n=\{x\in\R_k^{n+2}:x_{n+1}=x_{n+2}=0\} $$

$$ B_2(t)=(0,\dots,0,\cos(t),\sin(t))\com{and} B_3(t)=(0,\dots,0,-\sin(t),\cos(t)) $$

then, the map $\phi:\tilde{\Sp}_{k}^{n-1}\times (a_1,a_2)\to \Sp_k^{n+1}$ given by

$$\phi(y,u)=r(u)\, y+\sqrt{1-r^2(u)} \, B_2(\theta(u))$$

defines a hypersurface with constant mean curvature h.

\end{thm}

\begin{proof}
The proof follows the same arguments as before. In this case

$$ (r^\prime)^2+r^2 \lambda^2=1-r^2$$

and

$$\<y,y\>=1,\quad\<B_2,B_2\>=1,\quad \<B_3,B_3\>=1 \com{and} \<\nu,\nu\>=1$$

where,

$$\nu=-r\lambda\,   y +\frac{r^2\, \lambda}{\sqrt{1-r^2}} \, B_2 + \frac{r^\prime}{\sqrt{1-r^2}} \, B_3$$

\end{proof}

%\begin{rem}\label{rem closed case S}

%As pointed out in \cite{P}, every solution $g(t)$  of the equation \ref{Sphere 1} can be extended to all real numbers, is periodic and satisfies %$g(t)<\sqrt{c}$. See also Lemma (\ref{solutions lemma}) to get this property. Analogously to the arguments used in \cite{P}, we have that besides the %families of immersed and embedded cmc hypersurfaces in the sphere, these examples include  families of immersed and embedded Lorentzian constant mean %curvature hypersurfaces in the de Sitter space.

%\end{rem}

%-----------------------------------------------------------------------
%-----------------------------------------------------------------------

\subsection{CMC hypersurfaces in $\Hi_k^{n+1}$}

Let us start this section by considering the following differential equations

\begin{eqnarray}\label{Hyperbolic m1}
(g^\prime(t))^2=q_2(g(t))\com{where} q_2(t)=c+t^2+t^2(h+t^{-n})^2
\end{eqnarray}

and
\begin{eqnarray}\label{Hyperbolic 1}
(g^\prime(t))^2=p_2(g(t))\com{where} p_1(t)=c+t^2-t^2(h+t^{-n})^2
\end{eqnarray}

where $c$ is a  constant.

\begin{thm}\label{example1 Hm1}
Let $c$ be a positive constant and let $g:(a_1,a_2)\to \R$ be a non constant positive solution of (\ref{Hyperbolic m1}). For any pair of integers $k$ and $n$  such that $n\ge 2$ and $2\le k\le n$, if we define

$$\lambda=h+g^{-n},\quad \mu=h-(n-1)g^{-n}\com{and} r=\frac{g}{\sqrt{c}}\quad \theta(t)=\int_0^t\frac{r(\tau)\lambda(\tau)}{r(\tau)^2+1}d\tau$$

and

$$\tilde{\Sp}_{k-2}^{n-1}=\Sp_k^{n+1}\cap\tilde{\R}_{k}^n\com{where}\tilde{\R}_{k-1}^n=\{x\in\R_k^{n+2}:x_{1}=x_{2}=0\} $$

$$ B_2(t)=(\cos(t),\sin(t),0,\dots,0)\com{and} B_3(t)=(-\sin(t),\cos(t),0,\dots,0) $$

then, the map $\phi:\tilde{\Sp}_{k-2}^{n-1}\times (a_1,a_2)\to \Sp_k^{n+1}$ given by

$$\phi(y,u)=r(u)\, y+\sqrt{1+r^2(u)} \, B_2(\theta(u))$$

defines a hypersurface with constant mean curvature h.

\end{thm}

\begin{proof}
The proof follows the same arguments as before. In this case

$$ (r^\prime)^2-r^2 \lambda^2=1+r^2$$

and

$$\<y,y\>=1,\quad\<B_2,B_2\>=-1,\quad \<B_3,B_3\>=-1 \com{and} \<\nu,\nu\>=-1$$

where,

$$\nu=-r\lambda\,   y -\frac{r^2\, \lambda}{\sqrt{1+r^2}} \, B_2 - \frac{r^\prime}{\sqrt{r^2+1}} \, B_3$$

\end{proof}

%------------------------------------------------------------
%------------------------------------------------------------

\begin{thm}\label{example2 Hm1}
Let $c$ be a negative constant and let $g:(a_1,a_2)\to \R$ be a non constant positive solution of (\ref{Hyperbolic m1}) such that $g(t)<\sqrt{-c}$. For any pair of integers $k$ and $n$  such that $n\ge 2$ and $2\le k\le n$, if we define

$$\lambda=h+g^{-n},\quad \mu=h-(n-1)g^{-n}\com{and} r=\frac{g}{\sqrt{-c}}\quad \theta(t)=\int_0^t\frac{r(\tau)\lambda(\tau)}{1-r(\tau)^2}d\tau$$

and

$$\tilde{\Hi}_{k-2}^{n-1}=\Hi_{k-1}^{n+1}\cap\tilde{\R}_{k-1}^n\com{where}\tilde{\R}_{k-1}^n=\{x\in\R_k^{n+2}:x_{k}=x_{k+1}=0\} $$

$$ B_2(t)=(0,\dots,0,\cosh(t),\sinh(t),0,\dots,0)\com{and} B_3(t)=(0,\dots,0,\sinh(t),\cosh(t),0,\dots,0) $$

then, the map $\phi:\tilde{\Hi}_{k-2}^{n-1}\times (a_1,a_2)\to \Hi_{k-1}^{n+1}$ given by

$$\phi(y,u)=r(u)\, y+\sqrt{1-r^2(u)} \, B_2(\theta(u))$$

defines a hypersurface with constant mean curvature h.

\end{thm}

\begin{proof}
The proof follows the same arguments as before. In this case

$$ (r^\prime)^2-r^2 \lambda^2=-1+r^2$$

and

$$\<y,y\>=-1,\quad\<B_2,B_2\>=-1,\quad \<B_3,B_3\>=1 \com{and} \<\nu,\nu\>=-1$$

where,

$$\nu=-r\lambda\,   y +\frac{r^2\, \lambda}{\sqrt{1-r^2}} \, B_2 - \frac{r^\prime}{\sqrt{1-r^2}} \, B_3$$

\end{proof}

\begin{thm}\label{rem example2 Hm1}
With the same notation as in the Theorem (\ref{example2 Hm1}), if $-1<r_0<1$ is any non zero constant, then the map $\phi:\tilde{\Hi}_{k-2}^{n-1}\times \R \to \Hi_{k-1}^{n+1}$ given by

$$\phi(y,t)=r_0\, y+\sqrt{1-r_0^2} \, B_2(t)$$

defines a complete hypersurface with constant mean curvature

$$ \frac{r_0}{n\, \sqrt{1-r_0^2}}\, +\, \frac{(n-1)\, \sqrt{1-r_0^2}}{n\, r_0}$$

In the case $k=2$, these examples are part of the examples known as hyperbolic cylinder in the anti de Sitter space $\Hi_1^{n+1}$.

\end{thm}

%------------------------------------------------------------
%------------------------------------------------------------

\begin{thm}\label{example3 Hm1}
Let $c$ be a negative constant and let $g:(a_1,a_2)\to \R$ be a non constant positive solution of (\ref{Hyperbolic m1}) such that $g(t)>\sqrt{-c}$. For any pair of integers $k$ and $n$  such that $n\ge 2$ and $2\le k\le n$, if we define

$$\lambda=h+g^{-n},\quad \mu=h-(n-1)g^{-n}\com{and} r=\frac{g}{\sqrt{-c}}\quad \theta(t)=\int_0^t\frac{r(\tau)\lambda(\tau)}{r(\tau)^2-1}d\tau$$

and

$$\tilde{\Hi}_{k-2}^{n-1}=\Hi_{k-1}^{n+1}\cap\tilde{\R}_{k-1}^n\com{where}\tilde{\R}_{k-1}^n=\{x\in\R_k^{n+2}:x_{k}=x_{k+1}=0\} $$

$$ B_2(t)=(0,\dots,0,\sinh(t),\cosh(t),0,\dots,0)\com{and} B_3(t)=(0,\dots,0,\cosh(t),\sinh(t),0,\dots,0) $$

then the map $\phi:\tilde{\Hi}_{k-2}^{n-1}\times (a_1,a_2)\to \Hi_{k-1}^{n+1}$ given by

$$\phi(y,u)=r(u)\, y+\sqrt{r^2(u)-1} \, B_2(\theta(u))$$

defines a hypersurface with constant mean curvature h.

\end{thm}

\begin{proof}
The proof follows the same arguments as before. In this case

$$ (r^\prime)^2-r^2 \lambda^2=-1+r^2$$

and

$$\<y,y\>=-1,\quad\<B_2,B_2\>=1,\quad \<B_3,B_3\>=-1 \com{and} \<\nu,\nu\>=-1$$

where,

$$\nu=-r\lambda\,   y -\frac{r^2\, \lambda}{\sqrt{r^2-1}} \, B_2 - \frac{r^\prime}{\sqrt{r^2-1}} \, B_3$$

\end{proof}

\begin{rem}\label{rem1 example3 Hm1}
When $k=2$, the previous theorem  provides a family of space like hypersurfaces in the anti de Sitter space, it is surprising that when we make $r=r_0$  constant, we do not obtain space like hypersurfaces but Lorentzian.

\end{rem}

%----------------------------------------------------------------------
%----------------------------------------------------------------------

%------------------------------------------------------------
%Example 1 H1
%------------------------------------------------------------

\begin{thm}\label{rem2 example3 H1}
With the same notation as in Theorem (\ref{example3 Hm1}), if $r_0$, with $r_0^2>1$,  is any constant, then the map $\phi:\tilde{\Hi}_{k-2}^{n-1}\times \R \to \Hi_{k-1}^{n+1}$ given by

$$\phi(y,t)=r_0\, y+\sqrt{r_0^2-1} \, B_2(t)$$

defines a cmc hypersurface with constant mean curvature

$$ \frac{r_0}{n\, \sqrt{r_0^2-1}}\, +\, \frac{(n-1)\, \sqrt{r_0^2-1}}{n\, r_0}$$

\end{thm}

%------------------------------------------------------------
%Example 2 H1
%------------------------------------------------------------

\begin{thm}\label{example4 H1}
Let $c$ be a negative constant and let $g:(a_1,a_2)\to \R $ be a non constant positive solution of (\ref{Hyperbolic 1}) such that $g(t)>\sqrt{-c}$. For any pair of integers $k$ and $n$  such that $n\ge 2$ and $1\le k\le n$, if we define

$$\lambda=h+g^{-n},\quad \mu=h-(n-1)g^{-n}\com{and} r=\frac{g}{\sqrt{-c}}\quad \theta(t)=\int_0^t\frac{r(\tau)\lambda(\tau)}{r(\tau)^2-1}d\tau$$

and

$$\tilde{\Hi}_{k-1}^{n-1}=\Hi_{k-1}^{n+1}\cap\tilde{\R}_{k}^n\com{where}\tilde{\R}_{k}^n=\{x\in\R_k^{n+2}:x_{n+1}=x_{n+2}=0\} $$

$$ B_2(t)=(0,\dots,0,\cos(t),\sin(t))\com{and} B_3(t)=(0,\dots,0,-\sin(t),\cos(t)) $$

then, the map $\phi:\tilde{\Hi}_{k-1}^{n-1}\times (a_1,a_2)\to \Hi_{k-1}^{n+1}$ given by

$$\phi(y,u)=r(u)\, y+\sqrt{r^2(u)-1} \, B_2(\theta(u))$$

defines a hypersurface with constant mean curvature h.

\end{thm}

\begin{proof}
The proof follows the same arguments as before. In this case

$$ (r^\prime)^2+r^2 \lambda^2=r^2-1$$

and

$$\<y,y\>=-1,\quad\<B_2,B_2\>=1,\quad \<B_3,B_3\>=1 \com{and} \<\nu,\nu\>=1$$

where,

$$\nu=-r\lambda\,   y -\frac{r^2\, \lambda}{\sqrt{r^2-1}} \, B_2 + \frac{r^\prime}{\sqrt{r^2-1}} \, B_3$$

\end{proof}

\begin{thm}\label{example5 H1}
Let $c$ be a positive constant and let $g:(a_1,a_2)\to \R $ be a non constant positive solution of (\ref{Hyperbolic 1}) For any pair of integers $k$ and $n$  such that $n\ge 2$ and $1\le k\le n$, if we define

$$\lambda=h+g^{-n},\quad \mu=h-(n-1)g^{-n}\com{and} r=\frac{g}{\sqrt{c}}\quad \theta(t)=\int_0^t\frac{r(\tau)\lambda(\tau)}{r(\tau)^2+1}d\tau$$

and

$$\tilde{\Sp}_{k-1}^{n-1}=\Sp_k^{n+1}\cap\tilde{\R}_{k-1}^n\com{where}\tilde{\R}_{k-1}^n=\{x\in\R_k^{n+2}:x_{k}=x_{k+1}=0\} $$

$$ B_2(t)=(0,\dots,0,\cosh(t),\sinh(t),0,\dots,0)\com{and} B_3(t)=(0,\dots,0,\sinh(t),\cosh(t),0,\dots,0) $$

then, the map $\phi:\tilde{\Sp}_{k-1}^{n-1}\times (a_1,a_2)\to \Hi_{k-1}^{n+1}$ given by

$$\phi(y,u)=r(u)\, y+\sqrt{r^2(u)+1} \, B_2(\theta(u))$$

defines a hypersurface with constant mean curvature h.

\end{thm}

\begin{proof}
The proof follows the same arguments as before. In this case

$$ (r^\prime)^2+r^2 \lambda^2=r^2+1$$

and

$$\<y,y\>=1,\quad\<B_2,B_2\>=-1,\quad \<B_3,B_3\>=1 \com{and} \<\nu,\nu\>=1$$

where,

$$\nu=-r\lambda\,   y -\frac{r^2\, \lambda}{\sqrt{r^2+1}} \, B_2 + \frac{r^\prime}{\sqrt{r^2+1}} \, B_3$$

\end{proof}

%------------------------------------------------------------
%------------------------------------------------------------
%---------------

\subsection{CMC hypersurfaces in $\R_k^{n+1}$}

Let us start this section by considering the following differential equations

\begin{eqnarray}\label{Euclidean m1}
(g^\prime(t))^2=q_3(g(t))\com{where} q_3(t)=c+t^2(h+t^{-n})^2
\end{eqnarray}

and
\begin{eqnarray}\label{Euclidean 1}
(g^\prime(t))^2=p_3(g(t))\com{where} p_3(t)=c-t^2(h+t^{-n})^2
\end{eqnarray}

where $c$ is a non zero constant.

\begin{thm}\label{example1 E1}
Let $c$ be a positive constant and let $g:(a_1,a_2)\to \R $ be a non constant positive solution of (\ref{Euclidean 1}). For any pair of integers $k$ and $n$  such that $n\ge 2$ and $0\le k\le n$, if we define

$$\lambda=h+g^{-n},\quad \mu=h-(n-1)g^{-n}\com{and} r=\frac{g}{\sqrt{c}}\quad R(t)=\int_0^t\, r(\tau)\lambda(\tau)\, d\tau$$

and

$$\tilde{\Sp}_{k}^{n-1}=\Sp_k^{n}\cap\tilde{\R}_{k}^{n}\com{where}\tilde{\R}_{k}^{n}=\{x\in\R_k^{n+1}:x_{n+1}=0\} $$

$$ B_2=(0,\dots,0,1) $$

then, the map $\phi:\tilde{\Sp}_{k}^{n-1}\times (a_1,a_2)\to \R_k^{n+1}$ given by

$$\phi(y,u)=r(u)\, y+ R(u) \, B_2$$

defines a hypersurface with constant mean curvature h.

\end{thm}

\begin{proof}
The proof follows the same arguments as before. In this case

$$ (r^\prime)^2+r^2 \lambda^2=1$$

and

$$\<y,y\>=1,\quad\<B_2,B_2\>=1, \com{and} \<\nu,\nu\>=1$$

where,

$$\nu=-r\lambda\,   y +r^\prime \, B_2 $$

\end{proof}

%\begin{rem}\label{rem closed case R}

%As pointed out in \cite{P}, every solution $g(t)$  of the equation \ref{Sphere 1} can be extended to all real numbers and  is periodic. See also %Lemma (\ref{solutions lemma 2}) to get this property. Analogously to the arguments used in \cite{P}, we have that besides the families of immersed %and embedded cmc hypersurfaces in the Euclidean space, these examples include  families of immersed and embedded Lorentzian constant mean curvature %hypersurfaces in the Minkowski space.

%\end{rem}

%--------------------------------------------------------------------
%--------------------------------------------------------------------

\begin{thm}\label{example2 E1}
Let $c$ be a positive constant and let $g:(a_1,a_2)\to \R $ be a non constant positive solution of (\ref{Euclidean 1}). For any pair of integers $k$ and $n$  such that $n\ge 2$ and $2\le k\le n$, if we define

$$\lambda=h+g^{-n},\quad \mu=h-(n-1)g^{-n}\com{and} r=\frac{g}{\sqrt{c}}\quad R(t)=\int_0^t\, r(\tau)\lambda(\tau)\, d\tau$$

and

$$\tilde{\Hi}_{k-2}^{n-1}=\Hi_{k-1}^{n}\cap\tilde{\R}_{k-1}^{n}\com{where}\tilde{\R}_{k-1}^{n}=\{x\in\R_k^{n+1}:x_{1}=0\} $$

$$ B_2=(1,0,\dots,0) $$

then, the map $\phi:\tilde{\Hi}_{k-2}^{n-1}\times (a_1,a_2)\to \R_k^{n+1}$ given by

$$\phi(y,u)=r(u)\, y+ R(u) \, B_2$$

defines a hypersurface with constant mean curvature h.

\end{thm}

\begin{proof}
The proof follows the same arguments as before. In this case

$$ (r^\prime)^2+r^2 \lambda^2=1$$

and

$$\<y,y\>=-1,\quad\<B_2,B_2\>=-1 \com{and} \<\nu,\nu\>=-1$$

where,

$$\nu=-r\lambda\,   y +r^\prime \, B_2 $$

\end{proof}

%--------------------------------------------------------------------
%--------------------------------------------------------------------

\begin{thm}\label{example3 E1}
Let $c$ be a positive constant and let $g:(a_1,a_2)\to \R $ be a non constant positive solution of (\ref{Euclidean m1}). For any pair of integers $k$ and $n$  such that $n\ge 2$ and $1\le k\le n$, if we define

$$\lambda=h+g^{-n},\quad \mu=h-(n-1)g^{-n}\com{and} r=\frac{g}{\sqrt{c}}\quad R(t)=\int_0^t\, r(\tau)\lambda(\tau)\, d\tau$$

and

$$\tilde{\Sp}_{k-1}^{n-1}=\Sp_k^{n}\cap\tilde{\R}_{k-1}^{n}\com{where}\tilde{\R}_{k-1}^{n}=\{x\in\R_k^{n+1}:x_{1}=0\} $$

$$ B_2=(1,0,\dots,0) $$

then, the map $\phi:\tilde{\Sp}_{k-1}^{n-1}\times (a_1,a_2)\to \R_k^{n+1}$ given by

$$\phi(y,u)=r(u)\, y+ R(u) \, B_2$$

defines a hypersurface with constant mean curvature h.

\end{thm}

\begin{proof}
The proof follows the same arguments as before. In this case

$$ (r^\prime)^2-r^2 \lambda^2=1$$

and

$$\<y,y\>=1,\quad\<B_2,B_2\>=-1 \com{and} \<\nu,\nu\>=-1$$

where,

$$\nu=-r\lambda\,   y -r^\prime \, B_2 $$

\end{proof}

\begin{rem}

For $k=1$, the examples above give space like immersion in the Minkowski space. None of these examples is complete.

\end{rem}

%--------------------------------------------------------------------
%--------------------------------------------------------------------

\begin{thm}\label{example4 E1}
Let $c$ be a negative constant and let $g:(a_1,a_2)\to \R $ be a non constant positive solution of (\ref{Euclidean m1}). For any pair of integers $k$ and $n$  such that $n\ge 2$ and $1\le k\le n$, if we define

$$\lambda=h+g^{-n},\quad \mu=h-(n-1)g^{-n}\com{and} r=\frac{g}{\sqrt{-c}}\quad R(t)=\int_0^t\, r(\tau)\lambda(\tau)\, d\tau$$

and

$$\tilde{\Hi}_{k-1}^{n-1}=\Hi_k^{n}\cap\tilde{\R}_{k}^{n-1}\com{where}\tilde{\R}_{k}^{n-1}=\{x\in\R_k^{n+1}:x_{n+1}=0\} $$

$$ B_2=(0,0,\dots,1) $$

then, the map $\phi:\tilde{\Hi}_{k}^{n-1}\times (a_1,a_2)\to \R_k^{n+1}$ given by

$$\phi(y,u)=r(u)\, y+ R(u) \, B_2$$

defines a hypersurface with constant mean curvature h.

\end{thm}

\begin{proof}
The proof follows the same arguments as before. In this case

$$ (r^\prime)^2-r^2 \lambda^2=-1$$

and

$$\<y,y\>=-1,\quad\<B_2,B_2\>=1 \com{and} \<\nu,\nu\>=-1$$

where,

$$\nu=-r\lambda\,   y -r^\prime \, B_2 $$

\end{proof}

\begin{rem}

For $k=1$, the examples above are  space like cmc immersions in the Minkowski space. In the next section we will show that they are complete.

\end{rem}

%--------------------------------------------------------------------
%--------------------------------------------------------------------

\section{Closed and complete examples.}

In this section we will analyze the completeness of the space-like examples and we will study if the semi riemannian  examples can be extended to closed hypersurfaces. 

\subsection{Closed and embedded examples with CMC in $\Sp_k^{n+1}$ }

\begin{thm}
Let $n\ge3$ be an integer. For any $h\in [-1,-\frac{2\sqrt{n-1}}{n})$ there exist solutions of the equation (\ref{Sphere m1}) such that the immersions given in Theorem (\ref{examples1 Sm1}) define  closed immersions. For $k=1$, these immersions define complete space like immersions in the de Sitter space.
\end{thm}

\begin{proof}

 We will show the theorem by showing that there exist positive values of $c$ such that the equation  (\ref{Sphere m1}) has periodic solutions $g(t)$ such that $g(t)>\sqrt{c}$. A direct verification shows that for values of $h$ in $h\in (-1,-\frac{2\sqrt{n-2}}{n})$, the polynomial $q_1$ has exactly two positive critical points $v_0$ and $v_1$ with $0<v_0<v_1$. The values for $v_0$ and $v_1$ are

$$ v_0=2^{-\frac{1}{n}} (\frac{h (n-2) + \sqrt{4 - 4 n + h^2 n^2}}{h^2-1})^\frac{1}{n}\com{and}
v_1 = 2^{-\frac{1}{n}} (\frac{h (n-2) - \sqrt{4 - 4 n + h^2 n^2}}{h^2-1})^\frac{1}{n} $$

It follows that the function $q_1$ is decreasing from $0$ to $v_0$, increasing from $v_0$ to $v_1$ and decreasing from $v_1$ to $\infty$. A direct computation shows that if $q_1(v_1)=0$ then

$$ c=c_1= (2 - 2 h^2)^\frac{n-2}{n}\,   n\,  {\big(}-h (n-2) + \sqrt{4 - 4n + h^2 n^2}{\big)}^{\frac{2-2n}{n}} (h^2 n-2 - h \sqrt{4 - 4n + h^2 n^2}) $$

It is not difficult to show that $c_1$ is positive. Taking in consideration the intervals where the function $q_1$ is increasing and decreasing, we have that, for every $c>c_1$ close to $c_1$ the polynomial $q_1$ has exactly $3$ positive roots $t_i(c)$ such that $t_0(c) <t_1(c)<v_1<t_2(c)$, and moreover,

$$q_1(t)>0\com{for} t_1(c)<t<t_2(c)\com{and} \lim_{t\to c_1^{+}}t_1(c)= \lim_{t\to c_1^{+}}t_2(c)=v_1 $$

Notice that if the conditions above hold true, then, any solution $g$  of the differential equation (\ref{Sphere m1}) is periodic and satisfy that

$$t_1(c)<g(t)<t_2(c)\com{for all} t$$

The expression $(\frac{v_1}{\sqrt{c_1}})^2$ reduces to

$$\frac{2 + (-2 + h^2) n + h \sqrt{4 + n (-4 + h^2 n)}}{2 (-1 + h^2) n}$$

and a direct computation shows that the expression above is greater than $1$ for values of $h\in(-1,-\frac{2\sqrt{n-2}}{n})$. Since $\frac{v_1}{\sqrt{c_1}}>1$, there exists a positive $\epsilon$ such that  for all $c\in(c_1,c_1+\epsilon)$ we have that $\frac{t_1(c)}{\sqrt{c}}>1$. Therefore, for all these $c$ we will have, using Lemma (\ref{solutions lemma 2}), that the solutions of the differential equation (\ref{Sphere m1}) are periodic and satisfy that $g(t)>\sqrt{c}$ and then, these functions will define closed immersions from $\Sp_{k-1}^{n-1}\times \R$ in $\Sp_k^{n+1}$. When $k=1$, the completeness of the riemannian induced metric on $\Sp^{n-1}\times \R$ easily follows using the fact that $g$ is positive and periodic. When $h=-1$, a direct verification shows that, under this additional condition, the polynomial $q_1$ satisfies that

 $$ \lim_{t\to\infty}q_1(t)=c \com{and} \lim_{t\to 0^+}q_1(t)=\infty$$

and it only has one critical point $v_0$ given by,

$$v_0={\big (}\frac{n-1}{n-2}{\big )}^{\frac{1}{n}}$$

Since the value of $q_1$ at this critical point is given by

$$q_1(v_0)=c-c_0\com{where}c_0= n\, (n-2)^{\frac{n-2}{n}}\,  ( n-1)^{\frac{2-2n}{n}}$$

Therefore, if $c\in (0,c_0)$, the polynomial $q_1$ will have exactly two positive roots $t_1(c)$ and $t_2(c)$ with  $t_1(c)<v_0<t_2(c)$. Since,

$$\frac{v_0}{\sqrt{c_0}}=\frac{n-1}{\sqrt{n^2-2n}}>1$$

 and $\lim_{c\to c_0}t_2(c)=v_0$ we get that for values of $c$ close to $c_0$, $t_2(c)$ is greater than $\sqrt{c}$ and therefore $g(t)>\sqrt{c}$ for all $t$. In this case, using a small variation of Lemma (\ref{solution lemma 1}), we get that the function $g(t)$ can be extended as an even function to all real numbers. Notice that in this case, the function $F$ defined in the proof Lemma (\ref{solution lemma 1}) grows linearly instead of  logarithmicaly. Since  $g(t)>\sqrt{c}$, then, we obtain closed examples with $h=-1$.

\end{proof}

\begin{rem}
In the previous examples, along the geodesic given by the principal direction $\mu$, as the arc length parameter $u$ of this geodesic goes to infinity,  the function $\lambda$ remains bounded when $h\ne-1$ and decays as  $u^{-n}$ when $h=-1$.
\end{rem}

\begin{thm}
For any $h\in (-\infty,-1)\cup(1,\infty)$, there exist solutions of the equation (\ref{Sphere m1}) such that the immersions given in Theorem (\ref{examples2 Sm1}) define  closed immersions. For $k=1$, these immersions define complete space like immersions in the de Sitter space. When $h\in (1,\infty)$ the examples are embedded.
\end{thm}

\begin{proof}

 We will show the theorem by showing that there exist positive values of $c$ such that the equation  (\ref{Sphere m1}) is defined in all real numbers bounded away from zero. A direct verification shows that for values of $h$ in $(-\infty,-1)\cup(1,\infty)$, the polynomial $q_1$ has exactly one positive critical point $v_0$  given by,

$$ v_0=2^{-\frac{1}{n}} (\frac{h (n-2) + \sqrt{4 - 4 n + h^2 n^2}}{h^2-1})^\frac{1}{n} $$

It follows that the function $q_1$ is decreasing from $0$ to $v_0$ and increasing from $v_0$ to $\infty$. A direct computation shows that, $q_1(v_0)=c+c_0$ where

$$ c_0= ( 2 h^2-2)^\frac{n-2}{n}\,   n\,  {\big(}h (-2 + n) + \sqrt{4 - 4 n + h^2 n^2)}{\big)}^{\frac{2-2n}{n}} (h^2 n-2 + h \sqrt{4 - 4n + h^2 n^2}) $$

It is not difficult to show that $c_0$ is positive. Taking in consideration the intervals where the function $q_1$ is increasing and decreasing, we have that, for every $c\in (-\infty,-c_0)$ the polynomial $q_1$ has exactly $2$ positive roots $t_i(c)$ such that $t_1(c)<v_0<t_2(c)$, and moreover,

$$q_1(t)>0\com{for} t>t_2(c) $$

Using the lemma (\ref{solution lemma 1}) we get that for values of $c$ in $ (-\infty,-c_0)$, any solution $g$  of the differential equation (\ref{Sphere m1}) can be extended to the whole real line, and moreover $g(t)\ge t_2(c)>0$. Therefore, using these solutions in Theorem (\ref{examples2 Sm1}) we obtain closed immersions of $\Hi_{k-1}^{n-1}\times \R$ in $\Sp_k^{n+1}$. When $k=1$, the completeness of the riemannian induced metric on $\Hi^{n-1}\times \R$ easily follows using the fact that $g$ is bounded away from zero.

\end{proof}

We also can repeat the argument used in \cite{P} to obtain the following theorem.

\begin{thm}
For any $n\ge 2$  and  for any integer $m>1$ and $h$ between the numbers

$$\cot{\frac{\pi}{m}} \com{and} \frac{ (m^2-2)\, \sqrt{(n-1)}}{n \sqrt{m^2-1}} $$

there are examples of the type (\ref{examples1 S1}) that represent embedded hypersurfaces of $\Sp_k^{n-1}\times\Sp^1$ in $\Sp_k^{n+1}$ with constant mean curvature $h$, such that
the group of isometries contains the group $O(n)\times Z_m$.

\end{thm}

\subsection{Closed and embedded examples with CMC in $\Hi_k^{n+1}$ }

In this section we point out that the theorems regarding embedding of hypersurfaces of the hyperbolic space proven in \cite{P} and \cite{P1}, can be extended to the semi riemannian case.

\begin{thm}

For any $n\ge 2$  and any  $h\ge0$ there are examples of the type (\ref{example5 H1}) that represent embedded hypersurfaces of $\Sp_{k-1}^{n-1}\times\R$ in $\Hi_{k-1}^{n+1}$ with constant mean curvature $h$. Moreover, if
$h>1$,  the embedded hypersurfaces admit the group $O(n)\times Z$ in its group of isometries, where $Z$ is the group of integers.
\end{thm}

\begin{thm}\label{case n immersion}

For any $n\ge 2$  and any  $h\ge0$ there are examples of the type (\ref{example4 H1}) that represent embedded hypersurfaces of $\Hi_{k-1}^{n-1}\times\R$ in $\Hi_{k-1}^{n+1}$ with constant mean curvature $h$. Moreover, there exists a constant $h_0(n)<-1$ that depends of $n$ such that for any  $h<h_0(n)$, there are embedded examples of the type   (\ref{example4 H1}) that induce hypersurfaces with constant mean curvature $h$ of $\Hi_{k-1}^{n-1}\times\Sp^1$ in $\Hi_{k-1}^{n+1}$.

\end{thm}

\begin{rem}
In \cite{P1} we provide a way  to compute $h_0(n)$ numerically.
\end{rem}

%--------------------------------------------------------------------
%--------------------------------------------------------------------

\subsection{Closed and embedded examples with CMC in $\R_k^{n+1}$ }

\begin{thm}
For any $h\ne0$ and any $c<0$, there exist solutions of the equation (\ref{Euclidean m1}) such that the immersions given in Theorem (\ref{example4 E1}) define closed immersions. For $k=1$, these immersions define complete space like immersions in the Minkowski space.
\end{thm}

\begin{proof}

The proof follows the same type of arguments as the previous ones. In this case, since $h\ne0$, the polynomial $q_3$ given in (\ref{Euclidean m1}) satisfies that,

$$\lim_{t\to\infty}q_3(t)=\lim_{t\to 0}q_3(t)=\infty\com{and only has one positive critical point}$$

Since $c<0$, $q_3$ has exactly two positive roots $t_1(c)<t_2(c)$ and $q_3(t)>0$ for values of $t>t_2(c)$. Therefore, by lemma (\ref{solution lemma 1}), there exists an even solution defined in all the real numbers with global minimum equal to  $t_2(c)$. Similar arguments as those in the previous theorems complete the proof of the theorem.

\end{proof}

The same proof for theorem (7.2) in \cite{P} generalizes to the following result,

\begin{thm}
For  any  real number  $h$  and any positive $c$ the solutions of (\ref{Euclidean 1}) are periodic. Moreover, the immersions given in example (\ref{example1 E1}) are closed immersions of $\Sp_k^k\times \R$ in $\R_k^{n+1}$, and the immersions given in example (\ref{example2 E1}) are closed immersions of  $\Hi_k^{n-1}\times\R$ in $\R_k^{n+1}$. Additionally, when $h>0$, both types of  examples are embedded.

\end{thm}

%--------------------------------------------------------------------
%--------------------------------------------------------------------

\end{document}